\documentclass[12pt,a4paper,reqno]{amsart}
\usepackage[T1]{fontenc}
\usepackage[utf8]{inputenc}
\usepackage{lmodern}
\usepackage{amsmath,amssymb,amsthm}
\usepackage[margin=27mm,includeheadfoot]{geometry}
\usepackage[expansion=false]{microtype}
\usepackage{hyperref}
\allowdisplaybreaks[1]
\numberwithin{equation}{section}
\newtheorem{theorem}{Theorem}[section]
\newtheorem{proposition}[theorem]{Proposition}
\newtheorem{lemma}[theorem]{Lemma}
\newtheorem{corollary}[theorem]{Corollary}
\theoremstyle{remark}
\newtheorem{remark}{Remark}[section]
\newcommand{\T}{\mathbb T}
\newcommand{\R}{\mathbb R}
\newcommand{\Z}{\mathbb Z}
\newcommand{\N}{\mathbb N}
\newcommand{\W}{\mathcal W}
\newcommand{\Maj}{\mathcal M}
\newcommand{\Min}{\mathfrak m}
\DeclareMathOperator{\dist}{dist}

\title[Maximal estimates for higher-degree Weyl sums]{Improved $L^p$ maximal estimates
for higher-degree Weyl sums}

\author{Xuezhi Chen}
\address{School of Mathematics and Physics, University of Science and Technology Beijing, Beijing 100083, China}
\email{chen\_xuezhi@ustb.edu.cn}

\author{Changxing Miao}
\address{School of Mathematics and Physics, University of Science and Technology Beijing, Beijing 100083, China}
\email{miao\_changxing@ustb.edu.cn}

\author{Jiye Yuan}
\address{School of Science, China University of Geosciences (Beijing), Beijing 100083, China}
\email{yuan\_jiye@cugb.edu.cn}

\author{Tengfei Zhao}
\address{School of Mathematics and Physics, University of Science and Technology Beijing, Beijing 100083, China}
\email{zhao\_tengfei@ustb.edu.cn}

\date{}
\subjclass[2020]{11L15, 11P55, 42B25}
\keywords{Weyl sums, maximal estimates, minor arcs, Vinogradov mean value theorem, complete exponential sums}
\begin{document}
\begin{abstract}
For each integer $k\ge3$, we establish the following estimate for all
$1\le p<\infty$ and $\varepsilon>0$:
\[
\left\|\sup_{t\in\T}\left|\sum_{n=1}^N e(nx+n^kt)\right|\right\|_{L^p(\T)}
\ll_{k,p,\varepsilon}
N^{\max\{1-1/K_k,\,1-1/p\}+\varepsilon},
\]
where $K_k=\min\{2^{k-1},k(k+1)/2\}$, $e(z)=e^{2\pi iz}$, and $\T=\R/\Z$. For $k\ge5$, this improves our earlier bound~\cite{chen-miao-yuan-zhao} for $1\le p<D_k$, where $D_k=\min\{2^{k-1},k(k-1)\}$, and, together with the corresponding lower bound, extends the sharp range up to $N^\varepsilon$ from $p\ge D_k$ to $p\ge K_k$. The main new ingredient is a mixed minor-arc moment estimate obtained before taking the time supremum. The minor-arc maximal estimate then follows from a smooth reproducing kernel argument, while the major arcs are handled directly by Poisson summation and classical complete exponential-sum estimates. The reduction from suprema to mean values also yields a short and direct proof of improved $L^p$ estimates for maximal operators on Weyl sums with general polynomial phases, strengthening~\cite[Theorem~1.1]{baker-chen-shparlinski}.
\end{abstract}
\maketitle

\section{Introduction}\label{sec:introduction}
Let $\T=\R/\Z$, with Haar measure of total mass one, and write
$e(z)=\exp(2\pi iz)$. Fix an integer $k\ge3$. For $N\in\N$, define
\begin{equation}\label{eq:definition}
 S_N(x,t)=\sum_{n=1}^N e(nx+n^kt),\qquad
 \W_N(x)=\sup_{t\in\T}|S_N(x,t)|.
\end{equation}
We suppress the fixed degree $k$ in $S_N$ and $\W_N$, and put
\[
 K_k=\min\left\{2^{k-1},\frac{k(k+1)}2\right\}.
\]

For the quadratic phase, Barron~\cite{barron} proved
\[
 \left\|\sup_{t\in\T}\left|\sum_{n=1}^N e(nx+n^2t)\right|\right\|_{L^4(\T)}
 \ll_\varepsilon N^{3/4+\varepsilon},
\]
which is sharp up to $N^\varepsilon$. His proof uses induction on scale
and a version of parabolic rescaling. It also exploits the bounded
overlap of the spatial projections of rectangles selected around
maximizing points.
Baker~\cite{baker-quadratic} subsequently determined the order of magnitude
of the corresponding $L^p$ norm for every $1\le p<\infty$; at $p=4$,
the order is $N^{3/4}(\log N)^{1/4}$.
His upper-bound argument restricts attention to large values and uses
their simultaneous rational approximation structure, followed by
estimates for complete Gauss sums and oscillatory integrals.
Miao, Yuan and Zhao~\cite{miao-yuan-zhao-quadratic} extended these
quadratic maximal estimates to $\T^d$ for $d\ge2$, obtaining bounds
that are sharp up to $N^\varepsilon$, with Barron further reducing the
$N^\varepsilon$ loss to a power of $\log N$ through an alternative proof
in the appendix.

Baker, Chen and Shparlinski~\cite{baker-chen-shparlinski} studied
maximal operators on Weyl sums with general polynomial phases.
To state their results and fix notation for later use, let $\Gamma$ be a
nonempty proper subset of
$\{1,\ldots,k\}$, and put
\begin{equation}\label{eq:group-parameters}
 \begin{gathered}
 \Delta=\{1,\ldots,k\}\setminus \Gamma,\qquad m=|\Gamma|,\qquad
 s_k=\frac{k(k+1)}2,\\
 \tau_\Gamma=\sum_{j\in \Gamma}j,\qquad \sigma_\Delta=\sum_{j\in \Delta}j=s_k-\tau_\Gamma.
 \end{gathered}
\end{equation}
Write $\T^\Gamma=\prod_{j\in \Gamma}\T$ and similarly for $\T^\Delta$.
For $\boldsymbol\alpha=(\alpha_1,\ldots,\alpha_k)\in\T^k$, write
$\boldsymbol\alpha_\Gamma=(\alpha_j)_{j\in\Gamma}\in\T^\Gamma$ and
$\boldsymbol\alpha_\Delta=(\alpha_j)_{j\in\Delta}\in\T^\Delta$. We define
\begin{equation}\label{eq:group-definition}
 \begin{aligned}
 F_N(\boldsymbol\alpha)
 &=\sum_{n=1}^N e\!\left(\sum_{j=1}^k\alpha_jn^j\right),\\
 V_{\Gamma,N}(\boldsymbol\alpha_\Gamma)
 &=\sup_{\boldsymbol\alpha_\Delta\in\T^\Delta}|F_N(\boldsymbol\alpha)|,\qquad
 \mathsf M_{\Gamma,p}(N)=\|V_{\Gamma,N}\|_{L^p(\T^\Gamma)}.
 \end{aligned}
\end{equation}
For every $1\le p<\infty$ and $\varepsilon>0$, the estimate
in~\cite[Theorem~1.1]{baker-chen-shparlinski}, together with norm
monotonicity for $p<2s_k+k-m$, gives
\begin{equation}\label{eq:intro-bcs-mean}
 \mathsf M_{\Gamma,p}(N)
 \ll_{k,p,\varepsilon}
 N^{\max\{1-\tau_\Gamma/(2s_k+k-m),\,1-\tau_\Gamma/p\}+\varepsilon}.
\end{equation}
The matching lower bound $\mathsf M_{\Gamma,p}(N)\gg_{k,p}N^{1-\tau_\Gamma/p}$
shows that the exponent is sharp up to $N^\varepsilon$ for
$p\ge2s_k+k-m$. The estimate
in~\cite[Theorem~1.2]{baker-chen-shparlinski} also gives
\begin{equation}\label{eq:intro-bcs-weyl}
 \mathsf M_{\Gamma,p}(N)
 \ll_{k,p,\varepsilon}
 N^{\max\{1-2^{1-k},\,1-\tau_\Gamma/p\}+\varepsilon}.
\end{equation}
The latter has a smaller exponent than~\eqref{eq:intro-bcs-mean}
when $2^{k-1}\tau_\Gamma<2s_k+k-m$ and $1\le p<2s_k+k-m$.
Their arguments use mean value estimates and pointwise Weyl sum bounds,
together with the rational approximation structure of large values and
bounds for complete exponential sums.

The preceding general results also yield bounds for the Weyl sums
in~\eqref{eq:definition}, but the special form of the phase permits
further improvements. In our earlier
work~\cite[Theorem~1.1]{chen-miao-yuan-zhao}, we proved that
\begin{equation}\label{eq:previous}
 \|\W_N\|_{L^p(\T)}
 \ll_{k,p,\varepsilon}
 N^{\max\{1-1/D_k,\,1-1/p\}+\varepsilon},\qquad
 D_k=\min\{2^{k-1},k(k-1)\},
\end{equation}
for $k\ge3$ and $1\le p<\infty$.
The two entries in $D_k$ come from pointwise Weyl sum estimates:
$2^{k-1}$ corresponds to classical Weyl differencing, whereas $k(k-1)$
comes from Vinogradov mean value theory; see
\cite[Remark~1.5]{chen-miao-yuan-zhao}.
The argument in that work refines the rational approximation of large values using
the power-free and power-full structure of the denominator, obtains
level-set bounds, and then integrates these bounds to estimate the
$L^p$ norm~\cite[Sections~2--3]{chen-miao-yuan-zhao}.

The bounds in~\cite{chen-miao-yuan-zhao} are sharp up to $N^\varepsilon$
for $k=3$ and, for $k\ge4$, in the range $D_k\le p<\infty$.
Our main result improves the upper bounds for $k\ge5$ as follows:

\begin{theorem}\label{thm:main}
Let $k\ge3$ be an integer. For every $1\le p<\infty$ and
$\varepsilon>0$,
\begin{equation}\label{eq:main}
 \|\W_N\|_{L^p(\T)}
 \ll_{k,p,\varepsilon}
 N^{\max\{1-1/K_k,\,1-1/p\}+\varepsilon}.
\end{equation}
\end{theorem}

We compare this estimate with our earlier result and discuss its sharpness.

\begin{remark}\label{rem:main-comparison}
\textup{(i)} For $k\ge5$, $K_k<D_k$, so~\eqref{eq:main} improves
\eqref{eq:previous} when $1\le p<D_k$ and has the same exponent when
$p\ge D_k$.

\noindent\textup{(ii)} The lower bound established in our earlier
work~\cite[Theorem~1.3 and Section~4]{chen-miao-yuan-zhao} gives
\begin{equation}\label{eq:lower}
 \|\W_N\|_{L^p(\T)}\gg_{k,p} N^{1-1/p}
 \qquad(4<p<\infty).
\end{equation}
For $k\ge4$, combining~\eqref{eq:lower} with~\eqref{eq:main} shows
that the upper bound is sharp up to $N^\varepsilon$ for $K_k\le p<\infty$.
For $k\ge5$, this extends the sharp range from $p\ge D_k$ to $p\ge K_k$.
\end{remark}

The cases $k=3,4$ follow from~\eqref{eq:previous}.
We now outline the proof of Theorem~\ref{thm:main} for $k\ge5$, so that
$K_k=k(k+1)/2$. Motivated by the rational approximation structure of
large values in~\cite{chen-miao-yuan-zhao}, we estimate the time supremum
separately on the major and minor arcs.
Adapting Wooley's translation and completion
argument~\cite{wooley,wooley-corrigendum} and applying the Vinogradov
mean value theorem~\cite{bourgain-demeter-guth,wooley-nested}, we obtain
\begin{equation}\label{eq:intro-moment}
 \int_{\Min(X)}\int_{\T}|S_X(x,t)|^{2K_k}\,dx\,dt
 \ll_{k,\varepsilon}X^{2K_k-k-2+\varepsilon},
\end{equation}
where $\Min(X)$ is defined in~\eqref{eq:wooley-arcs}.

We work with smooth Weyl sums to control the tails after Poisson
summation by repeated integration by parts. Partial summation transfers
\eqref{eq:intro-moment} to these sums on the common minor-arc set
$\Min_0(M)$. Enlarging the major-arc set $\Maj_0(M)$ to
$\Maj_\tau(M)$ then provides the separation needed for a reproducing
kernel in $t$ to yield the
minor-arc maximal estimate.

On $\Maj_\tau(M)$, we combine Poisson summation with the pointwise
estimate for complete sums in~\cite[Lemma~2.4]{baker-chen-shparlinski}.
This gives major-arc maximal estimates for every $r>2k$, which suffice
for Theorem~\ref{thm:main}. Appendix~\ref{app:fourth} gives a refinement
based on the estimates of Cochrane and Zheng~\cite{cochrane-zheng}.
For $k\ge4$, this refinement extends the range
to $r\ge4$ by interpolation, with an $M^\varepsilon$ loss for sufficiently small
$\tau=\tau(\varepsilon)>0$; in particular,
\begin{equation}\label{eq:intro-major-fourth}
 \left\|\sup_{t\in\Maj_\tau(M)}|S_M^w(\cdot,t)|\right\|_{L^4(\T)}
 \ll_{k,w,\varepsilon}M^{3/4+\varepsilon}.
\end{equation}
The exponent $3/4$ is sharp up to $M^\varepsilon$, but this refinement
does not improve the global $L^4$ bound obtained by combining the arc
estimates for $k\ge5$, which remains limited by the minor-arc exponent
$1-1/K_k$.

Combining the major- and minor-arc estimates at $r=K_k$ and choosing
$\tau$ sufficiently small in terms of $\varepsilon$, we obtain the global
$L^{K_k}$ bound for $\W_N$ by Fourier completion and a dyadic decomposition.
For $1\le p<K_k$, the desired estimate follows from the monotonicity of
$L^p$ norms on $\T$ with normalized measure. For $K_k<p<\infty$, it follows
by interpolating the $L^{K_k}$ estimate with the trivial bound
$\|\W_N\|_{L^\infty(\T)}\le N$.

Applying the reduction from suprema to mean values in this paper, we obtain
the following estimate for the general polynomial
maximal operators defined in~\eqref{eq:group-definition}, with the
notation and assumptions of~\eqref{eq:group-parameters}.

\begin{theorem}\label{thm:group-maximal}
Let $k\ge3$ be an integer. For every $1\le p<\infty$ and
$\varepsilon>0$,
\begin{equation}\label{eq:group-maximal}
 \mathsf M_{\Gamma,p}(N)
 \ll_{k,p,\varepsilon}
 N^{\max\{1-\tau_\Gamma/(2s_k),\,1-\tau_\Gamma/p\}+\varepsilon}.
\end{equation}
\end{theorem}

The lower bound proved in~\cite[Section~3.2]{baker-chen-shparlinski}
shows that the exponent in~\eqref{eq:group-maximal} is sharp up to
$N^\varepsilon$ for $p\ge2s_k$.

\begin{remark}\label{rem:group-comparison}
\textup{(i)} Compared with~\eqref{eq:intro-bcs-mean}, the bound
in~\eqref{eq:group-maximal} replaces $2s_k+k-m$ by $2s_k$ and strictly improves the exponent for
$1\le p<2s_k+k-m$. The argument in Section~\ref{sec:extensions} gives
a short and direct proof of this improvement
of~\cite[Theorem~1.1]{baker-chen-shparlinski}.

\noindent\textup{(ii)} Combining the improved inverse theorem discussed in
\cite[Section~9.1]{hu-short-intervals} with the large-value argument of
\cite[Lemma~2.9 and Section~3.3]{baker-chen-shparlinski} replaces
$2^{k-1}$ in~\eqref{eq:intro-bcs-weyl} by $D_k$ and yields
\[
 \mathsf M_{\Gamma,p}(N)
 \ll_{k,p,\varepsilon}
 N^{\max\{1-1/D_k,\,1-\tau_\Gamma/p\}+\varepsilon}.
\]
The bound in~\eqref{eq:group-maximal} is at least as strong when
$D_k\tau_\Gamma\ge2s_k$.
\end{remark}

The proof of Theorem~\ref{thm:group-maximal} uses a product of
one-dimensional Dirichlet kernels and the Vinogradov mean value theorem. The reproducing kernel argument extends
that of Section~\ref{subsec:smooth-minor} to groups of polynomial coefficients.
When $1\in \Delta$, the exponent $1-\tau_\Gamma/(2s_k)$ also appears in the
almost-everywhere estimate of Chen and
Shparlinski~\cite[Theorem~2.3]{chen-shparlinski-new-bounds}.

The paper is organized as follows.
Section~\ref{sec:mixed} establishes the mixed minor-arc estimate using
Wooley's translation argument and the Vinogradov mean value theorem.
Section~\ref{sec:proof} establishes the minor- and major-arc maximal estimates
and combines them to prove Theorem~\ref{thm:main}.
Section~\ref{sec:extensions} extends the reproducing kernel argument to general
polynomial phases and proves Theorem~\ref{thm:group-maximal}.
Appendix~\ref{app:fourth} uses the estimates of Cochrane and Zheng to
bound fourth moments of complete sums and obtain a stronger $L^4$
major-arc estimate.

We write $\|u\|_{\T}=\min_{j\in\Z}|u-j|$ for the distance to the nearest
integer, and $U\ll V$ to mean $|U|\le CV$ for some $C>0$.
Implicit constants may depend on $k$, fixed smooth functions and fixed
parameters in the statement, but are independent of the summation length
and rational denominators. All integrals over tori use normalized Haar
measure. For $f\in L^1(\R)$, we use the Fourier transform convention
$\widehat f(\xi)=\int_{\R}f(u)e(-u\xi)\,du$.
For $f\in L^1(\T)$, the same notation denotes its Fourier coefficients,
$\widehat f(m)=\int_{\T}f(u)e(-mu)\,du$ for $m\in\Z$.
The supremum of a nonnegative function over the empty set is defined
to be zero.


\subsection*{AI-assisted preparation}
OpenAI's ChatGPT (GPT-5.6 Sol) and Codex (GPT-6 Astra) were used
to assist with mathematical exploration,
literature searches, exposition, and \LaTeX{} editing during the preparation
of this manuscript. The authors have verified all proofs and references
and revised the language and structure of the manuscript. They take full
responsibility for its final content.

\section{A mixed moment estimate on the minor arcs}\label{sec:mixed}
Throughout Sections~\ref{sec:mixed} and~\ref{sec:proof}, we assume
$k\ge5$, so that $K_k=k(k+1)/2$.
For an integer $X\ge2$, let
\begin{equation}\label{eq:wooley-arcs}
 \Min(X)=\left\{t\in\T:
 \|qt\|_{\T}>X^{1-k}
 \text{ for every integer }1\le q\le X\right\}.
\end{equation}
This definition is similar to that in~\cite[Section~2]{wooley}.
For $\boldsymbol a\in\T^k$ and a positive integer $s$, define
\begin{equation}\label{eq:vmvt-definition}
 f_X(\boldsymbol a)=\sum_{n=1}^X e(a_1n+\cdots+a_kn^k),
 \qquad
 J_{s,k}(X)=\int_{\T^k}|f_X(\boldsymbol a)|^{2s}\,d\boldsymbol a.
\end{equation}
We use the translation argument of Wooley~\cite[Section~2]{wooley}
with the interval completion in~\cite{wooley-corrigendum}, retaining
the integration in the linear coefficient. The following proof estimates
the mixed moment directly. The translation-kernel bound follows from
the spacing of the translated phases implied by the minor-arc condition.

\begin{proposition}\label{prop:mixed}
For $k\ge5$, $s\in\N$ and integers $X\ge2$,
\begin{equation}\label{eq:mixed-general}
 \int_{\Min(X)}\int_{\T}|S_X(x,t)|^{2s}\,dx\,dt
 \ll_{k,s}
 X^{K_k-k-2}(\log(2X))^{2s+1}J_{s,k}(2X).
\end{equation}
\end{proposition}
\begin{proof}
Put
\[
 I_s(X)=\int_{\Min(X)}\int_{\T}|S_X(x,t)|^{2s}\,dx\,dt.
\]
For $\boldsymbol n\in\{1,\ldots,X\}^{2s}$, write
$\sigma_j(\boldsymbol n)=\sum_{i=1}^s(n_i^j-n_{s+i}^j)$.
Expanding the integrand and interchanging the finite sum and the
integrals, we obtain
\[
 I_s(X)
 =\sum_{\boldsymbol n\in\{1,\ldots,X\}^{2s}}
 \left(\int_{\T}e\!\left(x\sigma_1(\boldsymbol n)\right)\,dx\right)
 \left(\int_{\Min(X)}e\!\left(t\sigma_k(\boldsymbol n)\right)\,dt\right).
\]
Since $|\sigma_j(\boldsymbol n)|\le sX^j$, orthogonality in the missing
coefficients $a_2,\ldots,a_{k-1}$ gives
\[
 1=\int_{\T^{k-2}}
 e\!\left(\sum_{j=2}^{k-1}a_j\sigma_j(\boldsymbol n)\right)
 \prod_{j=2}^{k-1}H_j(a_j)\,da_2\cdots da_{k-1},
\]
where
\[
 H_j(u)=\sum_{\substack{h\in\Z\\|h|\le sX^j}}e(-hu)
 \qquad(2\le j\le k-1).
\]
Inserting this identity into the preceding expression for $I_s(X)$ gives
\begin{equation}\label{eq:fourier-identity}
 I_s(X)=\int_{\mathcal U}|f_X(\boldsymbol a)|^{2s}
 \prod_{j=2}^{k-1}H_j(a_j)\,d\boldsymbol a,
\end{equation}
where $\boldsymbol a=(a_1,\ldots,a_k)=(x,a_2,\ldots,a_{k-1},t)$ and
$\mathcal U=\T^{k-1}\times\Min(X)$.

Estimating~\eqref{eq:fourier-identity} directly by
$|H_j(u)|\le2sX^j+1$ gives
\[
 \begin{aligned}
 I_s(X)
 &\le\left(\prod_{j=2}^{k-1}(2sX^j+1)\right)
 \int_{\T^k}|f_X(\boldsymbol a)|^{2s}\,d\boldsymbol a\\
 &\ll_{k,s}X^{\sum_{j=2}^{k-1}j}J_{s,k}(X)
 =X^{K_k-k-1}J_{s,k}(X).
 \end{aligned}
\]
For $s=K_k$, the Vinogradov mean value theorem therefore yields only
$I_{K_k}(X)\ll_{k,\varepsilon}X^{2K_k-k-1+\varepsilon}$ for every
$\varepsilon>0$, one power of $X$ short of~\eqref{eq:intro-moment}.

We use the minor-arc condition to improve this estimate by averaging
over translations of the summation variable.
After the corresponding change of coefficient variables, the leading
coefficient $t$ remains unchanged and the argument of $H_{k-1}$ depends
linearly on the translation parameter, with slope $kt$.

Denote $P_{\boldsymbol a}(z)=\sum_{j=1}^ka_jz^j$.
For each integer $1\le y\le X$, the change of summation variable
$m=n+y$ gives
\[
 f_X(\boldsymbol a)
 =\sum_{m=y+1}^{y+X}e\!\left(P_{\boldsymbol a}(m-y)\right).
\]
We expand the translated phase as
\[
 P_{\boldsymbol a}(z-y)
 =P_{\boldsymbol a}(-y)+\sum_{j=1}^k\xi_jz^j,
\]
where the coefficients $\xi_j$ and the inverse change of variables are
given by
\begin{equation}\label{eq:coefficient-translation}
 \begin{aligned}
 \xi_j&=\sum_{\ell=j}^k\binom{\ell}{j}(-y)^{\ell-j}a_\ell,
 &a_j&=\sum_{\ell=j}^k\binom{\ell}{j}y^{\ell-j}\xi_\ell
 \qquad(1\le j\le k).
 \end{aligned}
\end{equation}
Consequently,
\begin{equation}\label{eq:shifted-weyl-sum}
 f_X(\boldsymbol a)
 =e(P_{\boldsymbol a}(-y))
 \sum_{m=y+1}^{y+X}e\!\left(\sum_{j=1}^k\xi_jm^j\right).
\end{equation}
Thus $\xi_k=a_k=t$ and $a_{k-1}=\xi_{k-1}+kyt$.
In particular, the factor $H_{k-1}(a_{k-1})$ becomes
$H_{k-1}(\xi_{k-1}+kyt)$, whose value varies with $y$.

For each integer $1\le y\le X$, the change of variables
\eqref{eq:coefficient-translation} maps $\mathcal U$ onto itself and
preserves measure.\footnote{For fixed $y$ and $t$, the map on $\T^{k-1}$
is affine, with an upper triangular integer linear part having diagonal
entries one and an integer inverse given by~\eqref{eq:coefficient-translation}.
It is therefore invertible and preserves Haar measure. Since $\xi_k=t$,
it also preserves $\mathcal U$.}
We may therefore apply it in~\eqref{eq:fourier-identity}, writing
$\boldsymbol\xi=(\xi_1,\ldots,\xi_k)\in\mathcal U$ for the new
integration variable.

The summation interval in~\eqref{eq:shifted-weyl-sum} still depends on $y$. To obtain a common
expression that can be averaged in $y$, we embed all these intervals
in $\{1,\ldots,2X\}$ and recover each one by Fourier orthogonality.
Let
\[
 \mathcal D_X(\gamma)=\sum_{n=1}^Xe(-\gamma n),\qquad
 \boldsymbol e_1=(1,0,\ldots,0).
\]
For every integer $m$,
\[
 \int_\T e\!\left(\gamma(m-y)\right)\mathcal D_X(\gamma)\,d\gamma
 =\begin{cases}
 1,&y+1\le m\le y+X,\\
 0,&\text{otherwise}.
 \end{cases}
\]
Applying this identity to the shifted sum gives the completion formula
\begin{equation}\label{eq:translated-completion}
 f_X(\boldsymbol a)
 =e(P_{\boldsymbol a}(-y))
 \int_\T f_{2X}(\boldsymbol\xi+\gamma\boldsymbol e_1)
 e(-y\gamma)\mathcal D_X(\gamma)\,d\gamma.
\end{equation}
Substituting~\eqref{eq:translated-completion}
into~\eqref{eq:fourier-identity}, making the coefficient change
\eqref{eq:coefficient-translation}, and taking absolute values, we obtain,
for every $1\le y\le X$,
\begin{equation}\label{eq:mixed-completed}
 \begin{aligned}
 I_s(X)
 &=\left|\int_{\T^{2s}}\Phi_X(\boldsymbol\gamma)
 e\!\left(-y\sum_{i=1}^s(\gamma_i-\gamma_{s+i})\right)
 \int_{\mathcal U}G(\boldsymbol\xi,\boldsymbol\gamma)
 \mathcal K_y(\boldsymbol\xi)\,d\boldsymbol\xi\,d\boldsymbol\gamma\right|\\
 &\le\int_{\T^{2s}}|\Phi_X(\boldsymbol\gamma)|
 \int_{\mathcal U}|G(\boldsymbol\xi,\boldsymbol\gamma)|
 |\mathcal K_y(\boldsymbol\xi)|\,d\boldsymbol\xi\,d\boldsymbol\gamma,
 \end{aligned}
\end{equation}
where $\boldsymbol\gamma=(\gamma_1,\ldots,\gamma_{2s})\in\T^{2s}$ and
\[
 \begin{aligned}
 \Phi_X(\boldsymbol\gamma)
 &=\prod_{i=1}^s\mathcal D_X(\gamma_i)\overline{\mathcal D_X(\gamma_{s+i})},\\
 G(\boldsymbol\xi,\boldsymbol\gamma)
 &=\prod_{i=1}^s
 f_{2X}(\boldsymbol\xi+\gamma_i\boldsymbol e_1)
 \overline{f_{2X}(\boldsymbol\xi+\gamma_{s+i}\boldsymbol e_1)},\\
 \mathcal K_y(\boldsymbol\xi)
 &=\prod_{j=2}^{k-1}
 H_j\!\left(\sum_{\ell=j}^k
 \binom{\ell}{j}y^{\ell-j}\xi_\ell\right).
 \end{aligned}
\]
Note that both $\Phi_X$ and $G$ are independent of $y$.
Averaging this inequality over $1\le y\le X$ therefore replaces
$|\mathcal K_y(\boldsymbol\xi)|$ by
$X^{-1}\sum_{y=1}^X|\mathcal K_y(\boldsymbol\xi)|$.
We estimate this average using the minor-arc condition.

The bounds
\[
 |H_j(u)|\le2sX^j+1,\qquad
 |H_{k-1}(u)|\ll_s
 \min\{X^{k-1},\|u\|_\T^{-1}\}
\]
yield
\begin{equation}\label{eq:averaged-kernel}
 \begin{aligned}
 \frac1X\sum_{y=1}^X|\mathcal K_y(\boldsymbol\xi)|
 &\ll_{k,s}X^{\sum_{j=2}^{k-2}j}\cdot
 \frac1X\sum_{y=1}^X|H_{k-1}(\xi_{k-1}+kyt)|\\
 &\ll_{k,s}X^{\sum_{j=2}^{k-2}j}\cdot\Psi_X(t,\xi_{k-1}),
 \end{aligned}
\end{equation}
where
\[
 \Psi_X(t,\alpha)
 =\frac1X\sum_{y=1}^X
 \min\{X^{k-1},\|kyt+\alpha\|_\T^{-1}\}.
\]
We now prove the uniform estimate
\begin{equation}\label{eq:translation-kernel-bound}
 \sup_{\substack{t\in\Min(X)\\\alpha\in\T}}\Psi_X(t,\alpha)
 \ll_k X^{k-2}\log(2X).
\end{equation}
Partition $\{1,\ldots,X\}$ into consecutive blocks of length
$L=\lfloor X/k\rfloor+1$, except possibly the last block, which
may be shorter. The number of blocks is $\lceil X/L\rceil\le k$.
For distinct $y,y'$ in the same block, we have
$|y-y'|\le L-1=\lfloor X/k\rfloor$, and hence
\[
 1\le k|y-y'|\le X,
 \qquad
 \|k(y-y')t\|_\T>X^{1-k}=:\rho.
\]
The second inequality follows directly from $t\in\Min(X)$.
Thus the phases $kyt+\alpha$ in each block are $\rho$-separated.
Therefore
\[
 \sum_{y\in\text{one block}}\min\{X^{k-1},\|kyt+\alpha\|_\T^{-1}\}
 \le X^{k-1}+\frac2\rho\sum_{j=1}^{\lfloor L/2\rfloor}\frac1j
 \ll_k X^{k-1}\log(2X).
\]
Summing over the blocks and dividing by $X$ proves
\eqref{eq:translation-kernel-bound}.\footnote{Compared with the trivial bound $\Psi_X(t,\alpha)\le X^{k-1}$,
averaging has therefore saved one power of $X$, up to a logarithmic
factor.}

Finally, H\"older's inequality and translation invariance in the linear coefficient give
\[
 \begin{aligned}
 \int_{\mathcal U}|G(\boldsymbol\xi,\boldsymbol\gamma)|\,d\boldsymbol\xi
 &\le\prod_{i=1}^{2s}
 \left(\int_{\T^k}
 |f_{2X}(\boldsymbol\xi+\gamma_i\boldsymbol e_1)|^{2s}
 \,d\boldsymbol\xi\right)^{1/(2s)}\\
 &=J_{s,k}(2X).
 \end{aligned}
\]
The geometric-sum bound gives
\[
 |\mathcal D_X(\gamma)|\ll\min\{X,\|\gamma\|_\T^{-1}\},
 \qquad \|\mathcal D_X\|_{L^1(\T)}\ll\log(2X).
\]
Since
$\|\Phi_X\|_{L^1(\T^{2s})}
=\|\mathcal D_X\|_{L^1(\T)}^{2s}\ll_s(\log(2X))^{2s}$,
averaging the inequality~\eqref{eq:mixed-completed} over $y$ and using
\eqref{eq:averaged-kernel}--\eqref{eq:translation-kernel-bound} gives
\[
 \begin{aligned}
 I_s(X)
 &\le\int_{\T^{2s}}|\Phi_X(\boldsymbol\gamma)|
 \int_{\mathcal U}|G(\boldsymbol\xi,\boldsymbol\gamma)|
 \frac1X\sum_{y=1}^X|\mathcal K_y(\boldsymbol\xi)|
 \,d\boldsymbol\xi\,d\boldsymbol\gamma\\
 &\ll_{k,s}
 \underbrace{X^{\sum_{j=2}^{k-2}j}}_{\text{lower-degree kernels}}\,
 \underbrace{X^{k-2}\log(2X)}_{\text{translation average}}\,
 \underbrace{(\log(2X))^{2s}}_{\text{completion}}\,
 J_{s,k}(2X).
 \end{aligned}
\]
The exponent is
\[
 \sum_{j=2}^{k-2}j+k-2
 =\frac{k(k-3)}2+k-2=K_k-k-2,
\]
as required.
\end{proof}

\begin{corollary}\label{cor:critical-moment}
For every $\varepsilon>0$ and every integer $X\ge2$,
\begin{equation}\label{eq:critical-moment}
 \int_{\Min(X)}\int_{\T}|S_X(x,t)|^{2K_k}\,dx\,dt
 \ll_{k,\varepsilon}X^{2K_k-k-2+\varepsilon}.
\end{equation}
\end{corollary}
\begin{proof}
Since $K_k=s_k=k(k+1)/2$ for $k\ge5$, the Vinogradov mean value theorem
\cite[Theorem~1.1]{bourgain-demeter-guth}
(see also~\cite[Corollary~1.3]{wooley-nested}) gives
\begin{equation}\label{eq:vmvt}
 J_{s,k}(X)\ll_{k,s,\eta}X^{s+\eta}+X^{2s-K_k+\eta}
 \qquad(\eta>0).
\end{equation}
Taking $s=K_k$ and $\eta=\varepsilon/2$ gives
\[
 J_{K_k,k}(2X)\ll_{k,\varepsilon}X^{K_k+\varepsilon/2},
 \qquad (\log(2X))^{2K_k+1}\ll_{k,\varepsilon}X^{\varepsilon/2}.
\]
Proposition~\ref{prop:mixed} therefore yields
\begin{align*}
 \int_{\Min(X)}\int_{\T}|S_X(x,t)|^{2K_k}\,dx\,dt
 &\ll_{k,\varepsilon}
 X^{K_k-k-2}X^{\varepsilon/2}X^{K_k+\varepsilon/2}\\
 &=X^{2K_k-k-2+\varepsilon}.\qedhere
\end{align*}
\end{proof}

\section{Proof of Theorem~\ref{thm:main}}\label{sec:proof}
In the major-arc analysis, we apply Poisson summation and control the
tails of the resulting series by repeated integration by parts.
We therefore work with smooth Weyl sums in both the major- and
minor-arc estimates.
Fix $w\in C_c^\infty((1/2,3))$ with $w=1$ on $[1,2]$, and for real
$M\ge2$ set
\begin{equation}\label{eq:smooth-sum}
 S_M^w(x,t)=\sum_{n\in\Z}w(n/M)e(nx+n^kt).
\end{equation}

On the minor arcs, we first transfer the mixed-moment bound of
Corollary~\ref{cor:critical-moment} to these smooth sums by partial
summation, and then use a reproducing kernel in the time variable to
obtain the corresponding maximal estimate. For these two steps, we
introduce the following enlarged major-arc sets.
Set $C_0=4^{k-1}$ and fix $\tau>0$
independently of $M$, with $\tau$ to be specified in~\eqref{eq:parameters}.
For $\sigma\in\{0,\tau\}$, put
\begin{equation}\label{eq:arcs}
 \Maj_\sigma(M)=
 \bigcup_{1\le q\le C_0M}\ 
 \bigcup_{\substack{0\le a<q\\(a,q)=1}}
 \left\{t\in\T:
 \left\|t-\frac aq\right\|_{\T}
 \le\frac{C_0M^{1-k+\sigma}}q\right\}.
\end{equation}
Write $\Min_\sigma(M)=\T\setminus\Maj_\sigma(M)$.
The enlargement to $\Maj_0(M)$ gives a common minor-arc set for all
lengths arising in partial summation, while the further enlargement to
$\Maj_\tau(M)$ provides the separation needed for the reproducing kernel
argument.

For sufficiently large $M$ and $M/2\le u\le3M$, set
$X=\lfloor u\rfloor$. Then $M/4\le X\le3M$, so the choice of $C_0$ gives
\[
 q\le X\ \Longrightarrow\ q\le C_0M,
 \qquad
 \frac{X^{1-k}}q\le\frac{C_0M^{1-k}}q.
\]
If $t\notin\Min(X)$, there exist $1\le q\le X$ and $a\bmod q$ with
$\|t-a/q\|_{\T}\le X^{1-k}/q$.
Reducing $a/q$ to lowest terms can only decrease the denominator and
increase the permitted arc width, so $t\in\Maj_0(M)$. Hence
\begin{equation}\label{eq:arc-inclusion}
\Min_0(M)\subseteq\Min(\lfloor u\rfloor)
 \qquad(M/2\le u\le3M).
\end{equation}
Furthermore, for every fixed $\tau>0$,
\begin{equation}\label{eq:separation}
 \dist\bigl(\Min_\tau(M),\Maj_0(M)\bigr)\gg M^{-k+\tau}
\end{equation}
for sufficiently large $M$, whenever both sets are nonempty.
To see this, let $t\in\Min_\tau(M)$ and let $u$ lie in the
$\Maj_0(M)$ arc about $a/q$. The triangle inequality for $\|\cdot\|_{\T}$ gives
\[
 \|t-u\|_{\T}>
 \frac{C_0M^{1-k}(M^\tau-1)}q
 \ge M^{-k}(M^\tau-1)\gg M^{-k+\tau}.
\]
Any bounded range of $M$ excluded here or in~\eqref{eq:arc-inclusion}
is absorbed into the implicit constants.

\subsection{The minor arcs}\label{subsec:smooth-minor}
Partial summation now gives the following smooth version of
Corollary~\ref{cor:critical-moment}.
\begin{lemma}\label{lem:smooth-moment}
For every $\eta>0$ and every real $M\ge2$, we have
\begin{equation}\label{eq:smooth-moment}
 \int_{\T}\int_{\Min_0(M)}|S_M^w(x,t)|^{2K_k}\,dt\,dx
 \ll_{k,w,\eta} M^{2K_k-k-2+\eta}.
\end{equation}
\end{lemma}
\begin{proof}
It suffices to consider sufficiently large $M$, since bounded $M$ are
covered by $|S_M^w|\ll_w M$.
Partial summation yields
\[
 S_M^w(x,t)=-\int_0^\infty S_{\lfloor u\rfloor}(x,t)
 \frac{w'(u/M)}M\,du,
\]
where $S_0=0$ and
\[
 \int_{M/2}^{3M}\frac{|w'(u/M)|}{M}\,du
 =\int_{1/2}^3|w'(v)|\,dv\ll_w1.
\]
By Minkowski's inequality,~\eqref{eq:arc-inclusion} and
Corollary~\ref{cor:critical-moment}, we obtain
\begin{align*}
 \|S_M^w\|_{L^{2K_k}(\T\times\Min_0(M))}
 &\le\int_{M/2}^{3M}\frac{|w'(u/M)|}{M}
 \|S_{\lfloor u\rfloor}\|_{L^{2K_k}(\T\times\Min(\lfloor u\rfloor))}\,du\\
 &\ll_{k,w,\eta}M^{(2K_k-k-2+\eta)/(2K_k)}.
\end{align*}
This proves~\eqref{eq:smooth-moment}.
\end{proof}

We next use a reproducing kernel in the time variable to pass from this
mixed-moment bound to a maximal estimate on $\Min_\tau(M)$.
\begin{lemma}\label{lem:minor-maximal}
For fixed $\tau>0$, every $\eta>0$ and $1\le r\le2K_k$,
\begin{equation}\label{eq:minor-maximal}
 \left\|\sup_{t\in\Min_\tau(M)}|S_M^w(\cdot,t)|\right\|_{L^r(\T)}
 \ll_{k,w,\tau,r,\eta}M^{1-1/K_k+\eta/(2K_k)}.
\end{equation}
\end{lemma}
\begin{proof}
Since $\T$ has normalized measure, H\"older's inequality reduces the proof
to the case $r=2K_k$.
Assume that $\Min_\tau(M)\ne\varnothing$.
We may also assume that $M$ is sufficiently large, since bounded $M$ are
covered by the trivial bound $|S_M^w|\ll_w M$.
Choose a Schwartz function $\psi$ whose Fourier transform equals one
on $[-1,1]$ and is supported in $[-2,2]$. For $\lambda\ge2$, define the
periodic reproducing kernel
\begin{align*}
 \mathcal P_\lambda(t)&=\sum_{j\in\Z}\lambda\psi(\lambda(t+j)),\\
 \widehat{\mathcal P}_\lambda(m)&=\widehat\psi(m/\lambda)=1
 \qquad(m\in\Z,\ |m|\le\lambda).
\end{align*}
Schwartz decay gives
\begin{align}
 \|\mathcal P_\lambda\|_{L^v(\T)}&\ll_v \lambda^{1-1/v}
 \qquad(1\le v\le\infty),\label{eq:kernel-norm}\\
 \int_{\|t\|_{\T}\ge d}|\mathcal P_\lambda(t)|\,dt
 &\ll_A(\lambda d)^{1-A}
 \qquad(A>1,\ \lambda^{-1}\le d\le1/2).\label{eq:kernel-tail}
\end{align}
Take $\lambda=(3M)^k$. Since
\[
 w(n/M)\ne0\ \Longrightarrow\ 0<n<3M
 \ \Longrightarrow\ \widehat{\mathcal P}_\lambda(n^k)=1,
\]
for $t_0\in\Min_\tau(M)$ we have
\begin{align*}
 S_M^w(x,t_0)
 &=\int_{\T}S_M^w(x,t)\mathcal P_\lambda(t_0-t)\,dt\\
 &=\int_{\Min_0(M)}S_M^w(x,t)\mathcal P_\lambda(t_0-t)\,dt
   +\int_{\Maj_0(M)}S_M^w(x,t)\mathcal P_\lambda(t_0-t)\,dt.
\end{align*}
If $\Maj_0(M)\ne\varnothing$, put
$d=\dist(\Min_\tau(M),\Maj_0(M))$. By~\eqref{eq:separation},
\[
 \lambda^{-1}\le d\le\tfrac12,\qquad \lambda d\gg M^\tau
\]
for sufficiently large $M$. Thus~\eqref{eq:kernel-tail} gives
\begin{align*}
 \left|\int_{\Maj_0(M)}S_M^w(x,t)\mathcal P_\lambda(t_0-t)\,dt\right|
 &\ll_w M\int_{\|u\|_{\T}\ge d}|\mathcal P_\lambda(u)|\,du\\
 &\ll_{w,A}M(\lambda d)^{1-A}
 \ll_{k,w,\tau,A}M^{1-\tau(A-1)}
 \ll_{k,w,\tau}M^{-100},
\end{align*}
where $A>1+101/\tau$ is fixed. If $\Maj_0(M)=\varnothing$, the
integral is zero.

Put
\[
 \widetilde S(x)=\int_{\Min_0(M)}|S_M^w(x,t)|^{2K_k}\,dt.
\]
H\"older's inequality and~\eqref{eq:kernel-norm} give
\begin{align*}
 \left|\int_{\Min_0(M)}S_M^w(x,t)\mathcal P_\lambda(t_0-t)\,dt\right|
 &\le \widetilde S(x)^{1/(2K_k)}\|\mathcal P_\lambda\|_{L^{2K_k/(2K_k-1)}(\T)}\\
 &\ll_k M^{k/(2K_k)}\widetilde S(x)^{1/(2K_k)}.
\end{align*}
Consequently, uniformly in $t_0\in\Min_\tau(M)$,
\begin{equation}\label{eq:pointwise-minor}
 |S_M^w(x,t_0)|
 \ll_{k,w,\tau} M^{k/(2K_k)}\widetilde S(x)^{1/(2K_k)}+M^{-100}.
\end{equation}
Taking the supremum in~\eqref{eq:pointwise-minor}, then the spatial
$L^{2K_k}(\T)$ norm, and finally using~\eqref{eq:smooth-moment}, we obtain
\begin{align*}
 &\left\|\sup_{t\in\Min_\tau(M)}|S_M^w(\cdot,t)|\right\|_{L^{2K_k}(\T)}\\
 &\quad\ll_{k,w,\tau}
 M^{k/(2K_k)}
 \left(\int_{\T}\int_{\Min_0(M)}
 |S_M^w(x,t)|^{2K_k}\,dt\,dx\right)^{1/(2K_k)}
 +M^{-100}\\
 &\quad\ll_{k,w,\tau,\eta}
 M^{k/(2K_k)}M^{(2K_k-k-2+\eta)/(2K_k)}+M^{-100}\\
 &\quad\ll M^{1-1/K_k+\eta/(2K_k)}.\qedhere
\end{align*}
\end{proof}

\subsection{The major arcs}
For integers $q\ge1$ and $a,b$ with $(a,q)=1$, define
\begin{equation}\label{eq:complete-sum}
 c_q(a,b)=\frac1q\sum_{u\bmod q}
 e\!\left(\frac{au^k+bu}{q}\right).
\end{equation}
The classical complete-sum estimate
\cite[Lemma~2.4]{baker-chen-shparlinski} gives
\begin{equation}\label{eq:complete-bound}
 |c_q(a,b)|\ll_k q^{-1/k}.
\end{equation}

Put
\begin{equation}\label{eq:integral}
 I_M(y,\beta)=\int_{\R}w(z/M)e(yz+\beta z^k)\,dz.
\end{equation}
For $x\in\R$, put $F(z)=w(z/M)e(xz+\beta z^k)$. Then
\[
 \widehat F(\xi)=\int_{\R}F(z)e(-z\xi)\,dz=I_M(x-\xi,\beta).
\]
Writing $n=u+qj$ and applying Poisson summation in $j$ gives
\begin{align}
 S_M^w(x,a/q+\beta)
 &=\sum_{u\bmod q}e(au^k/q)\sum_{j\in\Z}F(u+qj)\notag\\
 &=\sum_{m\in\Z}\left[\frac1q\sum_{u\bmod q}
 e\!\left(\frac{au^k+mu}{q}\right)\right]\widehat F(m/q)\notag\\
 &=\sum_{m\in\Z}c_q(a,m)I_M(x-m/q,\beta).\label{eq:poisson}
\end{align}
The right-hand side of~\eqref{eq:poisson} is $1$-periodic in $x$,
so the formula is well defined on $\T$.

For $I_M$, we will need the following two estimates later.
\begin{lemma}\label{lem:integral-estimates}
\textup{(i)} For every $r>2$,
\begin{equation}\label{eq:integral-moment}
 \int_{\R}\sup_{\beta\in\R}|I_M(y,\beta)|^r\,dy
 \ll_{k,w,r}M^{r-1}.
\end{equation}

\noindent\textup{(ii)} Fix $\tau>0$ and choose $C_1>2kC_0 3^{k-1}$.
Then, for every $A>0$,
\begin{equation}\label{eq:poisson-tail}
 \sum_{\substack{m\in\Z\\|x-m/q|\ge C_1M^\tau/q}}
 |I_M(x-m/q,\beta)|\ll_{A,k,w,\tau}M^{-A},
\end{equation}
uniformly for $1\le q\le C_0M$, $|\beta|\le C_0M^{1-k+\tau}/q$
and $x\in\R$. The same estimate holds with each summand multiplied by
$|c_q(a,m)|$, uniformly in $(a,q)=1$.
\end{lemma}
\begin{proof}
After the change of variables $z=Mu$, write
\[
 I_M(y,\beta)=M\int_{\R}w(u)e(\Phi(u))\,du,
 \qquad \Phi(u)=Myu+\beta M^ku^k.
\]
We first establish the uniform bound
\begin{equation}\label{eq:integral-decay}
 \sup_{\beta\in\R}|I_M(y,\beta)|
 \ll_{k,w}M(1+M|y|)^{-1/2}.
\end{equation}
For $M|y|\le1$, this follows from $|I_M(y,\beta)|\ll_w M$.
For $M|y|>1$, put $c_k=(2k3^{k-1})^{-1}$. On $[1/2,3]$,
\[
 \begin{cases}
 |\Phi'(u)|\ge\tfrac12M|y|,
   &\text{if }|\beta|M^k\le c_kM|y|,\\
 |\Phi''(u)|\gg_k M|y|,
   &\text{if }|\beta|M^k>c_kM|y|.
 \end{cases}
\]
Since $\Phi''(u)=k(k-1)\beta M^ku^{k-2}$ has constant sign on $[1/2,3]$
(and vanishes when $\beta=0$), $\Phi'$ is monotone there.
The first and second derivative estimates therefore give~\eqref{eq:integral-decay}.
Consequently, for $r>2$,
\[
 \int_{\R}\sup_{\beta\in\R}|I_M(y,\beta)|^r\,dy
 \ll_{k,w,r}M^r\int_{\R}(1+M|y|)^{-r/2}\,dy
 \ll_{k,w,r}M^{r-1},
\]
which proves~\eqref{eq:integral-moment}.

To prove~\eqref{eq:poisson-tail}, put $y=x-m/q$ and $\lambda=M|y|$.
In the stated range of $\beta$, the condition $|y|\ge C_1M^\tau/q$
and the choice of $C_1$ give, on $[1/2,3]$,
\[
 |\Phi'(u)|\ge\tfrac12\lambda,
 \qquad |\Phi^{(j)}(u)|\ll_{k,j}\lambda\quad(j\ge2).
\]
Repeated integration by parts, using $w\in C_c^\infty((1/2,3))$,
therefore yields, for every integer $B\ge1$,
\[
 |I_M(y,\beta)|\ll_{B,k,w}M(1+M|y|)^{-B}.
\]
For fixed $x$, as $m$ ranges over $\Z$, the points $M(x-m/q)$ form
a translated lattice with spacing $M/q$.
Put $\Lambda=C_1M^{1+\tau}/q$. On each of the two tails, we bound the
first lattice term separately and compare the remaining terms with an
integral, using the spacing $M/q$. Since $u\mapsto(1+u)^{-B}$ is
decreasing for $u\ge0$, summing the preceding bound gives, for every integer $B>1$,
\begin{align*}
 &\sum_{\substack{m\in\Z\\|x-m/q|\ge C_1M^\tau/q}}
 |I_M(x-m/q,\beta)|\\
 &\quad\ll_{B,k,w} M\left((1+\Lambda)^{-B}
       +\frac qM\int_\Lambda^\infty(1+u)^{-B}\,du\right)\\
 &\quad\ll_{B,k,w} M\left((1+\Lambda)^{-B}
       +\frac qM(1+\Lambda)^{1-B}\right)\\
 &\quad\ll_{B,k,w}M^{1-\tau(B-1)},
\end{align*}
where the last step uses $q\le C_0M$ and
$\Lambda\ge(C_1/C_0)M^\tau$.
Choosing an integer $B>1+(A+1)/\tau$ proves~\eqref{eq:poisson-tail}.
The weighted estimate follows from $|c_q(a,m)|\le1$.
\end{proof}

\begin{proposition}\label{prop:major-maximal}
For fixed $\tau>0$ and every real $r>2k$,
\begin{equation}\label{eq:major-maximal}
 \left\|\sup_{t\in\Maj_\tau(M)}|S_M^w(\cdot,t)|\right\|_{L^r(\T)}
 \ll_{k,w,r,\tau}M^{1-1/r+(1-1/r)\tau}.
\end{equation}
\end{proposition}
\begin{proof}
We first assume only that $r>2$ and fix $A>1$.
For $1\le q\le C_0M$, put
\[
 \mathcal R_q(x)=\{m\in\Z:|qx-m|<C_1M^\tau\},
 \qquad \#\mathcal R_q(x)\ll M^\tau.
\]
By~\eqref{eq:poisson} and~\eqref{eq:poisson-tail},
\[
 S_M^w(x,a/q+\beta)
 =\sum_{m\in\mathcal R_q(x)}c_q(a,m)I_M(x-m/q,\beta)
 +O_{k,w,\tau,r,A}(M^{-A/r}),
\]
uniformly for $x\in\R$, $(a,q)=1$ and
$|\beta|\le C_0M^{1-k+\tau}/q$.
H\"older's inequality, followed by taking the supremum in $a$ and
$\beta$ and extending the sum to all $m\in\Z$, gives
\[
\begin{aligned}
 &\sup_{\substack{a\bmod q,\ (a,q)=1\\
                  |\beta|\le C_0M^{1-k+\tau}/q}}
 |S_M^w(x,a/q+\beta)|^r\\
 &\quad\ll_{k,w,\tau,r,A}
 M^{(r-1)\tau}\sum_{m\in\Z}
 \left(\max_{\substack{a\bmod q\\(a,q)=1}}|c_q(a,m)|^r\right)
 \sup_{\beta\in\R}|I_M(x-m/q,\beta)|^r+M^{-A}.
\end{aligned}
\]
Writing $m=b+jq$ and using $c_q(a,b+jq)=c_q(a,b)$, we obtain
\begin{align}
 &\int_{\T}
 \sup_{\substack{a\bmod q,\ (a,q)=1\\
                  |\beta|\le C_0M^{1-k+\tau}/q}}
 |S_M^w(x,a/q+\beta)|^r\,dx\notag\\
 &\quad\ll
 M^{(r-1)\tau}
 \left(\sum_{b\bmod q}\max_{\substack{a\bmod q\\(a,q)=1}}
 |c_q(a,b)|^r\right)
 \int_{\R}\sup_{\beta\in\R}|I_M(y,\beta)|^r\,dy
 +O_A(M^{-A})\notag\\
 &\quad\ll
 M^{r-1+(r-1)\tau}
 \sum_{b\bmod q}\max_{\substack{a\bmod q\\(a,q)=1}}|c_q(a,b)|^r
 +O_A(M^{-A}),\label{eq:fixed-denominator}
\end{align}
where the last step uses~\eqref{eq:integral-moment}.

Now let $r>2k$. The complete-sum estimate~\eqref{eq:complete-bound} gives
\[
 \sum_{b\bmod q}\max_{\substack{a\bmod q\\(a,q)=1}}|c_q(a,b)|^r
 \ll_{k,r}q^{1-r/k}.
\]
By the definition of $\Maj_\tau(M)$, summing~\eqref{eq:fixed-denominator}
over $q\le C_0M$ therefore yields
\begin{align*}
 \int_{\T}\sup_{t\in\Maj_\tau(M)}|S_M^w(x,t)|^r\,dx
 &\ll M^{r-1+(r-1)\tau}\sum_{q\le C_0M}q^{1-r/k}
       +O_A(M^{1-A})\\
 &\ll_{k,w,\tau,r,A}M^{r-1+(r-1)\tau},
\end{align*}
since $\sum_{q\ge1}q^{1-r/k}<\infty$ for $r>2k$.
Taking $A=2$ and the $r$-th root proves~\eqref{eq:major-maximal}.
\end{proof}

\subsection{Completion of the proof}
The cases $k=3,4$ follow from our earlier
work~\cite[Theorem~1.1]{chen-miao-yuan-zhao}, so we assume $k\ge5$.
Then $K_k=k(k+1)/2>2k$. Fix $\varepsilon>0$ and set
\begin{equation}\label{eq:parameters}
 \tau=\eta=\min\{\varepsilon/8,1/16\}.
\end{equation}
Combining Lemma~\ref{lem:minor-maximal} and
Proposition~\ref{prop:major-maximal} with $r=K_k$ gives
\begin{align*}
 \left\|\sup_{t\in\T}|S_M^w(\cdot,t)|\right\|_{L^{K_k}(\T)}
 &\ll M^{1-1/K_k+\eta/(2K_k)}
       +M^{1-1/K_k+(1-1/K_k)\eta}\\
 &\ll M^{1-1/K_k+\eta}.
\end{align*}

We recover the original sums by Fourier completion.
For $M\le B\le2M$, define
\[
 \mathcal D_{M,B}(\xi)=\sum_{M<j\le B}e(-j\xi).
\]
Since $w=1$ on $[1,2]$, orthogonality gives
\[
 \sum_{M<n\le B}e(nx+n^kt)
 =\int_{\T}\mathcal D_{M,B}(\xi)S_M^w(x+\xi,t)\,d\xi.
\]
Geometric summation gives, uniformly in $B$,
\[
 |\mathcal D_{M,B}(\xi)|\ll\min\{M,\|\xi\|_{\T}^{-1}\},
 \qquad \|\mathcal D_{M,B}\|_{L^1(\T)}\ll\log(2M).
\]
Thus Minkowski's inequality and translation invariance yield
\begin{align*}
 \left\|\sup_{t\in\T}
 \left|\sum_{M<n\le B}e(nx+n^kt)\right|\right\|_{L^{K_k}(\T)}
 &\ll M^{1-1/K_k+\eta}\log(2M)\\
 &\ll_{k,w,\eta}M^{1-1/K_k+2\eta}.
\end{align*}
A dyadic decomposition, allowing a truncated final interval, now gives
\begin{align}
 \|\W_N\|_{L^{K_k}(\T)}
 &\ll_{k,\varepsilon}
 1+\sum_{\substack{j\ge1\\2^j\le N}}(2^j)^{1-1/K_k+2\eta}\notag\\
 &\ll N^{1-1/K_k+2\eta}
 \le N^{1-1/K_k+\varepsilon}.\label{eq:critical-maximal}
\end{align}
For $1\le p\le K_k$, this implies
\[
 \|\W_N\|_{L^p(\T)}\le\|\W_N\|_{L^{K_k}(\T)}
 \ll_{k,\varepsilon}N^{1-1/K_k+\varepsilon}.
\]
For $p>K_k$, the bound $\W_N\le N$ gives
\[
 \|\W_N\|_{L^p(\T)}
 \le N^{1-K_k/p}\|\W_N\|_{L^{K_k}(\T)}^{K_k/p}
 \ll_{k,\varepsilon}N^{1-1/p+(K_k/p)\varepsilon}
 \le N^{1-1/p+\varepsilon}.
\]
This proves Theorem~\ref{thm:main}.

\begin{remark}
The moment order $2K_k$ minimizes the exponent obtained by this argument.
Indeed, for any integer $s\ge1$, Proposition~\ref{prop:mixed},
\eqref{eq:vmvt} and partial summation give
\[
 \int_{\T}\int_{\Min_0(M)}|S_M^w(x,t)|^{2s}\,dt\,dx
 \ll_{k,w,s,\varepsilon}
 M^{K_k-k-2+\varepsilon}(M^s+M^{2s-K_k}).
\]
Applying the reproducing kernel argument of Lemma~\ref{lem:minor-maximal}
at moment order $2s$ yields, for every fixed $\tau>0$ and $1\le r\le2s$,
\[
 \left\|\sup_{t\in\Min_\tau(M)}|S_M^w(\cdot,t)|\right\|_{L^r(\T)}
 \ll_{k,w,s,r,\tau,\varepsilon}M^{\alpha(s)+\varepsilon},
 \qquad
 \alpha(s)=\max\left\{\frac12+\frac{K_k-2}{2s},\,1-\frac1s\right\}.
\]
The maximum is attained by the decreasing first branch for $s\le K_k$
and by the increasing second branch for $s\ge K_k$. Thus $\alpha(s)$ is
minimized at $s=K_k$, with value $1-1/K_k$. This calculation gives no sharpness statement for $p<K_k$.
The spatial exponent $r=K_k$ lies in both admissible ranges
$r\le2K_k$ for the minor arcs and $r>2k$ for the major arcs.
\end{remark}

\section{Extensions to general polynomial phases}\label{sec:extensions}
In this section, we prove Theorem~\ref{thm:group-maximal} using a product
of one-dimensional Dirichlet kernels and the Vinogradov mean value theorem.

Fix $\boldsymbol\alpha_\Gamma$ and put
$\mathcal G_{\boldsymbol\alpha_\Gamma}(\boldsymbol\alpha_\Delta)=|F_N(\boldsymbol\alpha)|^{2s_k}$.
This is a nonnegative trigonometric polynomial with coordinate frequencies
\[
 |h_j|\le s_kN^j=:L_j\qquad(j\in \Delta).
\]
Define the product Dirichlet kernel
\[
 \mathcal D_\Delta(\boldsymbol v)
 =\prod_{j\in \Delta}\left(\sum_{|h|\le L_j}e(hv_j)\right).
\]
Fourier orthogonality and the frequency support give
\[
 \mathcal G_{\boldsymbol\alpha_\Gamma}(\boldsymbol\alpha_\Delta)
 =\int_{\T^\Delta}\mathcal G_{\boldsymbol\alpha_\Gamma}(\boldsymbol v)
 \mathcal D_\Delta(\boldsymbol\alpha_\Delta-\boldsymbol v)\,d\boldsymbol v,
 \qquad
 \|\mathcal D_\Delta\|_{L^\infty(\T^\Delta)}
 \le\prod_{j\in \Delta}(2L_j+1)\ll_k N^{\sigma_\Delta}.
\]
Taking the supremum in $\boldsymbol\alpha_\Delta$ and using $\mathcal G_{\boldsymbol\alpha_\Gamma}\ge0$ gives
\begin{equation}\label{eq:group-kernel}
 V_{\Gamma,N}(\boldsymbol\alpha_\Gamma)^{2s_k}
 \ll_k N^{\sigma_\Delta}
 \int_{\T^\Delta}\mathcal G_{\boldsymbol\alpha_\Gamma}(\boldsymbol\alpha_\Delta)
 \,d\boldsymbol\alpha_\Delta.
\end{equation}

Integrating in $\boldsymbol\alpha_\Gamma$ and applying the Vinogradov mean value
theorem~\cite{bourgain-demeter-guth,wooley-cubic,wooley-nested}, we obtain, for every
$\delta>0$,
\begin{equation}\label{eq:group-critical-moment}
 \begin{aligned}
 \mathsf M_{\Gamma,2s_k}(N)^{2s_k}
 &\ll_k N^{\sigma_\Delta}
 \int_{\T^k}|F_N(\boldsymbol\alpha)|^{2s_k}\,d\boldsymbol\alpha\\
 &\ll_{k,\delta}N^{\sigma_\Delta+s_k+\delta}
 =N^{2s_k-\tau_\Gamma+\delta}.
 \end{aligned}
\end{equation}
For $1\le p\le2s_k$, normalized measure gives
\[
 \mathsf M_{\Gamma,p}(N)\le\mathsf M_{\Gamma,2s_k}(N)
 \ll_{k,\delta}N^{1-\tau_\Gamma/(2s_k)+\delta/(2s_k)}.
\]
For $p>2s_k$, use $V_{\Gamma,N}\le N$ to obtain
\[
 \mathsf M_{\Gamma,p}(N)^p
 \le N^{p-2s_k}\mathsf M_{\Gamma,2s_k}(N)^{2s_k}
 \ll_{k,\delta}N^{p-\tau_\Gamma+\delta}.
\]
Taking $\delta=\varepsilon$ proves Theorem~\ref{thm:group-maximal}.

\begin{remark}\label{rem:group-scope}
Since $\W_N(x)\le V_{\{1\},N}(x)$, Theorem~\ref{thm:group-maximal}
with $\Gamma=\{1\}$ gives a bound for $\W_N$ with constant exponent
$1-1/(2s_k)$. For $k\ge5$,
Theorem~\ref{thm:main} improves this to $1-1/s_k$ through the mixed
minor-arc estimate with the intermediate coefficients fixed at zero.
\end{remark}

\appendix
\section{Fourth moments of complete sums}\label{app:fourth}
In this appendix, we use classical complete-sum estimates to obtain a
stronger $L^4(\T)$ bound on the major arcs. Let $k\ge4$ and retain the
definitions of $S_M^w$, $\Maj_\tau(M)$ and $c_q$ in this range. Put
\begin{equation}\label{eq:fourth-energy}
 E(q)=\sum_{b\bmod q}
 \max_{\substack{a\bmod q\\(a,q)=1}}|c_q(a,b)|^4,
 \qquad E(1)=1.
\end{equation}

To estimate $E(q)$, we use the following bound of Cochrane and Zheng
for binomial exponential sums modulo prime powers.

\begin{theorem}[Cochrane and Zheng {\cite[Theorem~1.1]{cochrane-zheng}}]\label{thm:cochrane-zheng}
Let $p$ be a prime, let $n,m\ge2$ be integers, and let $a,b\in\Z$. Define
\[
 S(au^n+bu,p^m)
 =\sum_{u\bmod p^m}e\!\left(\frac{au^n+bu}{p^m}\right).
\]
Write $h=\operatorname{ord}_p(n-1)$, $\kappa=\operatorname{ord}_p(n)$ and
$\varrho=\operatorname{ord}_p(a)$, where $\operatorname{ord}_p$ denotes the
$p$-adic valuation. If $\varrho\le m-2$, then for $p>2$,
\[
 |S(au^n+bu,p^m)|
 \le \gcd(n-1,p-1)\,
 p^{\frac12\min\{1,\kappa\}}
 p^{\frac12\min\{h,m-2-\varrho\}}
 p^{m/2}\gcd(b,p^m)^{1/2},
\]
whereas for $p=2$,
\[
 |S(au^n+bu,p^m)|
 \le 2p^{\frac12\min\{h,m-2-\varrho\}}
 p^{m/2}\gcd(b,p^m)^{1/2}.
\]
\end{theorem}

We also record three classical estimates in the
normalization~\eqref{eq:complete-sum}.

\begin{lemma}[Weil bound {\cite[Lemma~2.2]{baker-chen-shparlinski}}]\label{lem:appendix-weil}
Let $p>k$ be a prime and let $a,b\in\Z$ with $p\nmid a$. Then
\[
 |c_p(a,b)|\le(k-1)p^{-1/2}.
\]
\end{lemma}

\begin{lemma}[{\cite[Lemma~2.3]{baker-chen-shparlinski}}]\label{lem:appendix-prime-power}
Let $p$ be a prime, let $\nu\ge2$ be an integer, and let $a,b\in\Z$
with $p\nmid a$. Then
\[
 |c_{p^\nu}(a,b)|\le(k-1)p^{-1}.
\]
\end{lemma}

\begin{lemma}[{\cite[Lemma~2.4]{baker-chen-shparlinski}}]\label{lem:appendix-complete}
Let $q\ge2$ be an integer and let $a,b\in\Z$ with $(a,q)=1$. Then
\[
 |c_q(a,b)|\ll_k q^{-1/k}.
\]
\end{lemma}

\begin{proposition}\label{prop:fourth-energy}
For every $\delta>0$,
\begin{equation}\label{eq:energy-series}
 \sum_{q=1}^{\infty}\frac{E(q)}{q^\delta}<\infty.
\end{equation}
Consequently, for $Q\ge2$,
\begin{equation}\label{eq:energy-sum}
 \sum_{q\le Q}E(q)\ll_{k,\delta}Q^\delta.
\end{equation}
\end{proposition}
\begin{proof}
Write $\Z_q=\Z/q\Z$ and
$\Z_q^\times=\{a\bmod q:(a,q)=1\}$.
Since $E(1)=1$, take coprime integers
$q_1,q_2\ge2$ and choose
$u_2q_2\equiv1\pmod{q_1}$ and $u_1q_1\equiv1\pmod{q_2}$.
The Chinese remainder theorem gives
\[
 c_{q_1q_2}(a,b)=c_{q_1}(au_2,bu_2)c_{q_2}(au_1,bu_1).
\]
Put $a_1\equiv au_2$, $b_1\equiv bu_2\pmod{q_1}$ and
$a_2\equiv au_1$, $b_2\equiv bu_1\pmod{q_2}$.
By the Chinese remainder theorem, $a\mapsto(a_1,a_2)$ is a bijection from
$\Z_{q_1q_2}^\times$ to $\Z_{q_1}^\times\times\Z_{q_2}^\times$.
Thus, for each fixed $b\in\Z_{q_1q_2}$,
\[
 \max_{a\in\Z_{q_1q_2}^\times}
 |c_{q_1q_2}(a,b)|^4
 =\prod_{i=1}^2
 \max_{a_i\in\Z_{q_i}^\times}
 |c_{q_i}(a_i,b_i)|^4.
\]
Likewise, $b\mapsto(b_1,b_2)$ is a bijection from $\Z_{q_1q_2}$ to
$\Z_{q_1}\times\Z_{q_2}$, so
\[
 E(q_1q_2)
 =\prod_{i=1}^2\left(
 \sum_{b_i\in\Z_{q_i}}
 \max_{a_i\in\Z_{q_i}^\times}
 |c_{q_i}(a_i,b_i)|^4\right)
 =E(q_1)E(q_2).
\]
Let $p$ be a prime and let $\nu\ge1$ be an integer.
By multiplicativity, it suffices to obtain suitable bounds for $E(p^\nu)$.
For $\nu=1$ and $p>k$, Lemma~\ref{lem:appendix-weil} gives
\[
 E(p)\le p\bigl((k-1)p^{-1/2}\bigr)^4=(k-1)^4p^{-1}.
\]
For the finitely many primes $p\le k$, the trivial bound gives $E(p)\le p$.

Now suppose $\nu\ge2$.
For $p\nmid a$, Lemmas~\ref{lem:appendix-prime-power}
and~\ref{lem:appendix-complete} give the uniform bound
\begin{equation}\label{eq:prime-power-uniform}
 |c_{p^\nu}(a,b)|\ll_k\min\{p^{-1},p^{-\nu/k}\}=p^{-H_\nu},
 \qquad H_\nu=\max\{1,\nu/k\}.
\end{equation}
To retain the dependence on $b$, apply Theorem~\ref{thm:cochrane-zheng}
with $n=k$, $m=\nu$ and $\varrho=0$.
Since
\[
 \gcd(k-1,p-1)\le k-1,\qquad
 p^{\operatorname{ord}_p(k)}\le k,\qquad
 p^{\operatorname{ord}_p(k-1)}\le k-1,
\]
division by $p^\nu$ gives, for every prime $p$,
\begin{equation}\label{eq:cz}
 |c_{p^\nu}(a,b)|
 \ll_k p^{-\nu/2}\gcd(b,p^\nu)^{1/2}.
\end{equation}

We group the residues $b$ by their greatest common divisor with $p^\nu$.
For $0\le j\le\nu$, set
\[
 B_j=\{b\bmod p^\nu:\gcd(b,p^\nu)=p^{\nu-j}\}.
\]
Then $B_0=\{0\}$ and $|B_j|\le p^j$. For $b\in B_j$,~\eqref{eq:cz}
gives $|c_{p^\nu}(a,b)|\ll_k p^{-j/2}$, uniformly for $p\nmid a$.
Combining~\eqref{eq:prime-power-uniform} with~\eqref{eq:cz}
and splitting the sum at $j=2H_\nu$ gives
\begin{align}
 E(p^\nu)
 &=\sum_{j=0}^\nu\sum_{b\in B_j}
 \max_{\substack{a\bmod p^\nu\\p\nmid a}}|c_{p^\nu}(a,b)|^4\notag\\
 &\ll_k\sum_{j=0}^\nu
 p^j\min\{p^{-4H_\nu},p^{-2j}\}\notag\\
 &\ll_k\sum_{0\le j\le2H_\nu}p^{j-4H_\nu}
       +\sum_{j>2H_\nu}p^{-j}
 \ll p^{-2H_\nu}
 =p^{-2\max\{1,\nu/k\}}.\label{eq:prime-power-energy}
\end{align}
Summing separately over $2\le\nu\le k$ and $\nu>k$ gives
\begin{equation}\label{eq:prime-power-tail}
 \sum_{\nu\ge2}E(p^\nu)
 \ll_k\sum_{\nu=2}^{k}p^{-2}
       +\sum_{\nu>k}p^{-2\nu/k}
 \ll_k p^{-2}.
\end{equation}
Combining the two cases, for every $\delta>0$ and $p>k$,
\[
 \begin{aligned}
 \sum_{\nu\ge1}\frac{E(p^\nu)}{p^{\nu\delta}}
 &=\frac{E(p)}{p^\delta}
   +\sum_{\nu\ge2}\frac{E(p^\nu)}{p^{\nu\delta}}\\
 &\ll_k p^{-1-\delta}+p^{-2-2\delta}.
 \end{aligned}
\]
The right-hand side is summable over $p$. For the finitely many remaining
primes $p\le k$, the bound $E(p)\le p$ and~\eqref{eq:prime-power-tail}
ensure convergence of the local series. Multiplicativity now gives
\[
 \sum_{q=1}^{\infty}\frac{E(q)}{q^\delta}
 =\prod_{p\ \mathrm{prime}}
 \left(1+\sum_{\nu\ge1}\frac{E(p^\nu)}{p^{\nu\delta}}\right)
 <\infty.
\]
Finally,
\[
 \sum_{q\le Q}E(q)
 \le Q^\delta\sum_{q\le Q}\frac{E(q)}{q^\delta}
 \ll_{k,\delta}Q^\delta.
\]
\end{proof}

The series in~\eqref{eq:energy-series} diverges when $\delta=0$. For any
$a\in\Z_q^\times$, orthogonality and Cauchy--Schwarz give
\[
 E(q)\ge\sum_{b\bmod q}|c_q(a,b)|^4
 \ge\frac1q\left(\sum_{b\bmod q}|c_q(a,b)|^2\right)^2
 =\frac1q,
\]
and hence $\sum_{q\le Q}E(q)\gg\log(2Q)$.

\begin{corollary}\label{cor:major-fourth}
For every fixed $\tau>0$ and every $\delta>0$,
\begin{equation}\label{eq:major-fourth}
 \left\|\sup_{t\in\Maj_\tau(M)}|S_M^w(\cdot,t)|\right\|_{L^4(\T)}
 \ll_{k,w,\tau,\delta}M^{3/4+3\tau/4+\delta/4}.
\end{equation}
\end{corollary}
\begin{proof}
The derivation of~\eqref{eq:fixed-denominator} uses Poisson summation and
Lemma~\ref{lem:integral-estimates}, both valid for $k\ge4$, and requires
only $r>2$. Taking $r=4$ and $A=2$, summing over $q\le C_0M$, and applying
Proposition~\ref{prop:fourth-energy}, we obtain
\[
 \int_{\T}\sup_{t\in\Maj_\tau(M)}|S_M^w(x,t)|^4\,dx
 \ll M^{3+3\tau}\sum_{q\le C_0M}E(q)+O(M^{-1})
 \ll M^{3+3\tau+\delta}.
\]
Taking fourth roots proves~\eqref{eq:major-fourth}.
\end{proof}
Interpolating~\eqref{eq:major-fourth} with the trivial $L^\infty$ bound
$O_w(M)$ gives
\[
 \left\|\sup_{t\in\Maj_\tau(M)}|S_M^w(\cdot,t)|\right\|_{L^r(\T)}
 \ll_{k,w,r,\tau,\delta}M^{1-1/r+(3\tau+\delta)/r},
 \qquad 4\le r<\infty.
\]
Taking $\tau,\delta>0$ with $3\tau+\delta\le4\varepsilon$
in~\eqref{eq:major-fourth} gives the bound $M^{3/4+\varepsilon}$.
The power $3/4$ is sharp up to $M^\varepsilon$. To see this, let
$\mathcal D=\mathcal D_{M,2M}$ be the completion kernel used in Section~\ref{sec:proof}.
Geometric summation gives $\|\mathcal D\|_{L^{4/3}(\T)}\ll M^{1/4}$, while
$w=1$ on $[1,2]$ and orthogonality give
\[
 \int_{\T}\mathcal D(x)S_M^w(x,0)\,dx
 =\lfloor2M\rfloor-\lfloor M\rfloor\asymp M.
\]
Since $0\in\Maj_\tau(M)$, H\"older's inequality yields
\begin{equation}\label{eq:major-fourth-lower}
 \left\|\sup_{t\in\Maj_\tau(M)}|S_M^w(\cdot,t)|\right\|_{L^4(\T)}
 \ge\|S_M^w(\cdot,0)\|_{L^4(\T)}\gg M^{3/4}.
\end{equation}

\end{document}